\documentclass[10pt,a4paper]{amsart}
\usepackage{amsfonts,amsmath,amssymb}
\usepackage{hyperref}
\usepackage[all]{xy}

\newtheorem*{theorem*}{Theorem}
\newtheorem{lemma}{Lemma}[subsection]
\newtheorem{proposition}[lemma]{Proposition}
\newtheorem{remark}[lemma]{Remark}

\newtheorem{theorem}[lemma]{Theorem}
\newtheorem{definition}[lemma]{Definition}
\newtheorem{notation}[lemma]{Notation}

\newtheorem{corollary}[lemma]{Corollary}

\baselineskip 18pt \textwidth 16cm  \theoremstyle{plain}

\newcommand{\tr}{\operatorname{Tr}}

\newcommand{\cc}{\mathbb{C}}

\newcommand{\eps}{\varepsilon}

\newcommand{\Z}{{\mathbb Z}}

\newcommand{\C}{{\mathbb C}}

\newcommand{\Fre}{{Fr\'{e}chet \,}}

\newcommand{\g}{{\mathfrak{g}}}
\newcommand{\h}{{\mathfrak{h}}}

\newcommand{\gd}{\g^{\sigma}}

\newcommand{\opi}{{\overline{\pi}}}
\newcommand{\de}{def}
\usepackage[usenames]{color}
\begin{document}

\author{Avraham Aizenbud}
\address{Avraham Aizenbud and Dmitry Gourevitch, Faculty of Mathematics
and Computer Science, The Weizmann Institute of Science POB 26,
Rehovot 76100, ISRAEL.} \email{aizenr@yahoo.com}
\author{Dmitry Gourevitch}
\email{guredim@yahoo.com}
\title{Some regular symmetric pairs}
\begin{abstract}
In \cite{AG2} we explored the question what symmetric pairs are
Gelfand pairs. We introduced the notion of regular symmetric pair
and conjectured that all symmetric pairs are regular. This
conjecture would imply that many symmetric pairs are Gelfand
pairs, including all connected symmetric pair over $\C$.

In this paper we show that the pairs $$(GL(V),O(V)), \,
(GL(V),U(V)), \, (U(V),O(V)), \, (O(V \oplus W),O(V) \times O(W)),
\, (U(V \oplus W),U(V) \times U(W))$$ are regular where $V$ and
$W$ are quadratic or hermitian spaces over arbitrary local field
of characteristic zero. We deduce from this that the pairs
$(GL_n(\C),O_n(\C))$ and $(O_{n+m}(\C),O_n(\C) \times O_m(\C))$
are Gelfand pairs.
\end{abstract}


%
%
%
%
%
%


\maketitle
 \tableofcontents

\section{Introduction}
In \cite{AG2} we explored the question what symmetric pairs are
Gelfand pairs. We introduced the notion of regular symmetric pair
and conjectured that all symmetric pairs are regular. This
conjecture would imply that many symmetric pairs are Gelfand
pairs, including all connected symmetric pair over $\C$.

In this paper we show that the pairs $$(GL(V),O(V)), \,
(GL(V),U(V)), \, (U(V),O(V)), \, (O(V \oplus W),O(V) \times O(W)),
\, (U(V \oplus W),U(V) \times U(W))$$ are regular where $V$ and
$W$ are quadratic or hermitian spaces over arbitrary local field
of characteristic zero. We deduce from this that the pairs
$(GL_n(\C),O_n(\C))$ and $(O_{n+m}(\C),O_n(\C) \times O_m(\C))$
are Gelfand pairs.

In general, if we would know that all symmetric pairs are regular,
then in order to show that a given symmetric pair $(G,H)$ is a
Gelfand pair it would be enough to check the following condition
that we called "goodness":\\
(*) Every closed $H$-double coset in $G$ is invariant with respect
to $\sigma$. Here, $\sigma$ is the anti-involution defined by
$\sigma(g):= \theta(g^{-1})$ and $\theta$ is an involution (i.e.
automorphism of order 2) of $G$ such that $H = G^{\theta}$.

This condition always holds for connected symmetric pairs over
$\C$.

Meanwhile, before the conjecture is proven, in order to show that
a given symmetric pair is a Gelfand pair one has to verify that
the pair is good, to prove that it is regular and also to compute
its "descendants" and show that they are also regular. The
"descendants" are certain symmetric pairs related to centralizers
of semisimple elements.

In this paper we develop further the tools from \cite{AG2} 
for proving regularity of symmetric pairs. We also
introduce a systematic way to compute descendants of classical
symmetric pairs.

Based on that we show that all the descendants of the above
symmetric pairs are regular.

\subsection{Structure of the paper} $ $\\
In section \ref{PrelNot} we introduce the notions that we discuss
in this paper. In subsection \ref{GelPairs} we discuss the notion
of Gelfand pair and review a classical technique for proving
Gelfand property due to Gelfand and Kazhdan. In subsection
\ref{SymPar} we review the results of \cite{AG2}, introduce the
notions of symmetric pair, descendants of a symmetric pair, good
symmetric pair and regular symmetric pair mentioned above and
discuss their relations to Gelfand property.

In section \ref{MainRes} we formulate the main results of the
paper. We also explain how they follow from the rest of the paper.

In section \ref{GradRepDef} we introduce terminology that enables
us to prove regularity for symmetric pairs in question.

In section \ref{SecReg} we prove regularity for symmetric pairs in
question. 

In section \ref{CompDes} we compute the descendants of those
symmetric pairs.

\subsection{Acknowledgements}
We are grateful to \textbf{Herve Jacquet} for a suggestion to
consider the pair $(U_{2n}, U_n \times U_n)$ which inspired this
paper. We also thank \textbf{Joseph Bernstein}, \textbf{Erez
Lapid}, \textbf{Eitan Sayag} and \textbf{Lei Zhang} for fruitful
discussions and \textbf{Gerard Schiffmann} for  useful remarks.

Both authors were partially supported by a BSF grant, a GIF grant, and an ISF Center
of excellency grant. A.A was also supported by ISF grant No. 583/09 and 
D.G. by NSF grant DMS-0635607. Any opinions, findings and conclusions or recommendations expressed in this material are those of the authors and do not necessarily reflect the views of the National Science Foundation.

\section{Preliminaries and notations} \label{PrelNot}
\setcounter{lemma}{0}
\begin{itemize}
\item Throughout the paper we fix an arbitrary local field $F$ of characteristic zero.
\item All the algebraic varieties and algebraic
groups that we will consider will be defined over $F$.
\item For a group $G$ acting on a set $X$ and an element $x \in X$
we denote by $G_x$ the stabilizer of $x$.
\item By a reductive group we mean an algebraic reductive group.
\end{itemize}

In this paper we will refer to distributions on algebraic
varieties over archimedean and non-archimedean fields. In the
non-archimedean case we mean the notion of distributions on
$l$-spaces from \cite{BZ}, that is linear functionals on the space
of locally constant compactly supported functions. In the
archimedean case one can consider the usual notion of
distributions, that is continuous functionals on the space of
smooth compactly supported functions, or the notion of Schwartz
distributions (see e.g. \cite{AG1}). It does not matter here which
notion to choose since in the cases of consideration of this
paper, if there are no nonzero equivariant Schwartz distributions
then there are no nonzero equivariant distributions at all (see
Theorem 4.0.2 in \cite{AG2}).

\begin{notation}
Let $E$ be an extension of $F$. Let $G$ be an algebraic group
defined over $F$. We denote by $G_{E/F}$ the canonical algebraic
group defined over $F$ such that $G_{E/F}(F)=G(E)$.
\end{notation}

\subsection{Gelfand pairs} \label{GelPairs}
$ $\\
In this section we recall a technique due to Gelfand and Kazhdan
(\cite{GK}) which allows to deduce statements in representation
theory from statements on invariant distributions. For more
detailed description see \cite{AGS}, section 2.

\begin{definition}
Let $G$ be a reductive group. By an \textbf{admissible
representation of} $G$ we mean an admissible representation of
$G(F)$ if $F$ is non-archimedean (see \cite{BZ}) and admissible
smooth \Fre representation of $G(F)$ if $F$ is archimedean.
\end{definition}

We now introduce three notions of Gelfand pair.

\begin{definition}\label{GPs}
Let $H \subset G$ be a pair of reductive groups.
\begin{itemize}
\item We say that $(G,H)$ satisfy {\bf GP1} if for any irreducible
admissible representation $(\pi,E)$ of $G$ we have
$$\dim Hom_{H(F)}(E,\cc) \leq 1$$

\item We say that $(G,H)$ satisfy {\bf GP2} if for any irreducible
admissible representation $(\pi,E)$ of $G$ we have
$$\dim Hom_{H(F)}(E,\cc) \cdot \dim Hom_{H}(\widetilde{E},\cc)\leq
1$$

\item We say that $(G,H)$ satisfy {\bf GP3} if for any irreducible
{\bf unitary} representation $(\pi,\mathcal{H})$ of $G(F)$ on a
Hilbert space $\mathcal{H}$ we have
$$\dim Hom_{H(F)}(\mathcal{H}^{\infty},\cc) \leq 1.$$
\end{itemize}

\end{definition}
Property GP1 was established by Gelfand and Kazhdan in certain
$p$-adic cases (see \cite{GK}). Property GP2 was introduced in
\cite{Gross} in the $p$-adic setting. Property GP3 was studied
extensively by various authors under the name {\bf generalized
Gelfand pair} both in the real and $p$-adic settings (see e.g.
\cite{vD,Bos-vD}).

We have the following straightforward proposition.

\begin{proposition}
$GP1 \Rightarrow GP2 \Rightarrow GP3.$
\end{proposition}

We will use the following theorem from \cite{AGS} which is a
version of a classical theorem of Gelfand and Kazhdan.

\begin{theorem}\label{DistCrit}
Let $H \subset G$ be reductive groups and let $\tau$ be an
involutive anti-automorphism of $G$ and assume that $\tau(H)=H$.
Suppose $\tau(\xi)=\xi$ for all bi $H(F)$-invariant distributions
$\xi$ on $G(F)$. Then $(G,H)$ satisfies GP2.
\end{theorem}

In the cases we consider in this paper GP2 is equivalent to GP1 by
the following proposition.

\begin{proposition} \label{GP2GP1}
$ $\\
(i) Let $V$ be a quadratic space (i.e. a linear space with a
non-degenerate quadratic form) and let $H \subset GL(V)$ be any
transpose invariant subgroup.
 Then $GP1$ is equivalent to $GP2$ for the pair
$(\mathrm{GL}(V),H)$.\\
(ii) Let $V$ be a quadratic space and let $H \subset O(V)$ be any
subgroup. Then $GP1$ is equivalent to $GP2$ for the pair
$(O(V),H)$.
\end{proposition}
It follows from the following 2 propositions.

\begin{proposition} \label{GKCor}
Let $H \subset G$ be reductive groups and let $\tau$ be an
anti-automorphism of $G$ such that\\
(i) $\tau^2\in Ad(G(F))$\\
(ii) $\tau$ preserves any closed conjugacy class in $G(F)$\\
(iii) $\tau(H)=H$.\\
Then $GP1$ is equivalent to $GP2$ for the pair $(G,H)$.
\end{proposition}

For proof see \cite{AG2}, Corollary 8.2.3.

\begin{proposition} $ $\\
(i) Let $V$ be a quadratic space and let $g \in GL(V)$. Then $g$
is conjugate to $g^{t}$.\\
(ii) Let $V$ be a quadratic space and let $g \in O(V)$. Then $g$
is conjugate to $g^{-1}$ inside $O(V)$.
\end{proposition}
Part (i) is well known. For the proof of (ii) see \cite{MVW},
Proposition I.2 in chapter 4.

\subsection{Symmetric pairs} \label{SymPar}
$ $\\
In this subsection we review some tools developed in \cite{AG2}
that enable to prove that a symmetric pair is a Gelfand pair. The
main results discussed in this subsection are Theorem
\ref{LinDes}, Theorem \ref{GoodHerRegGK} and Proposition
\ref{SpecCrit}.

\begin{definition}
A \textbf{symmetric pair} is a triple $(G,H,\theta)$ where $H
\subset G$ are reductive groups, and $\theta$ is an involution of
$G$ such that $H = G^{\theta}$. We call a symmetric pair
\textbf{connected} if $G/H$ is connected.

For a symmetric pair $(G,H,\theta)$ we define an antiinvolution
$\sigma :G \to G$ by $\sigma(g):=\theta(g^{-1})$, denote $\g:=Lie
G$, $\h := LieH$, $\gd:=\{a \in \g | \theta(a)=-a\}$. Note that
$H$ acts on $\gd$ by the adjoint action. Denote also
$G^{\sigma}:=\{g \in G| \sigma(g)=g\}$ and define a
\textbf{symmetrization map} $s:G \to G^{\sigma}$ by $s(g):=g
\sigma(g)$.

In case when the involution is obvious we will omit it.
\end{definition}

\begin{remark}
Let $(G,H,\theta)$ be a symmetric pair. Then $\g$ has a $\Z/2\Z$
grading given by $\theta$.
\end{remark}

\begin{definition}
Let $(G_1,H_1,\theta_1)$ and $(G_2,H_2,\theta_2)$ be symmetric
pairs. We define their \textbf{product} to be the symmetric pair
$(G_1 \times G_2,H_1 \times H_2,\theta_1 \times \theta_2)$.
\end{definition}

\begin{definition}
We call a symmetric pair $(G,H,\theta)$ \textbf{good} if for any
closed $H(F) \times H(F)$ orbit $O \subset G(F)$, we have
$\sigma(O)=O$.
\end{definition}

\begin{proposition} \label{GoodCrit}
Every connected symmetric pair over $\C$ is good.
\end{proposition}
For proof see e.g. \cite{AG2}, Corollary 7.1.7.

\begin{definition}
We say that a symmetric pair $(G,H,\theta)$ is a \textbf{GK pair}
if any $H(F) \times H(F)$ - invariant distribution on $G(F)$ is
$\sigma$ - invariant.
\end{definition}

\begin{remark}
Theorem \ref{DistCrit} implies that any GK pair satisfies GP2.
\end{remark}
\subsubsection{Descendants of symmetric pairs}
\begin{proposition} \label{PropDescend}
Let  $(G,H,\theta)$ be a symmetric pair. Let $g \in G(F)$ such
that $HgH$ is closed. Let $x=s(g)$. Then $x$ is a semisimple
element of $G$.
\end{proposition}
For proof see e.g. \cite{AG2}, Proposition 7.2.1.
\begin{definition}
In the notations of the previous proposition we will say that the
pair $(G_x,H_x,\theta|_{G_x})$ is a \textbf{descendant} of
$(G,H,\theta)$.
\end{definition}

\subsubsection{Tame symmetric pairs}
\begin{definition}
Let $\pi$ be an action of a reductive group $G$ on a smooth affine
variety $X$.
We say that an algebraic automorphism $\tau$ of $X$ is \textbf{$G$-admissible} if \\
(i) $\pi(G(F))$ is of index at most 2 in the group of
automorphisms of $X$
generated by $\pi(G(F))$ and $\tau$.\\
(ii) For any closed $G(F)$ orbit $O \subset X(F)$, we have
$\tau(O)=O$.
\end{definition}

\begin{definition}
We call an action of a reductive group $G$ on a smooth affine
variety $X$ \textbf{tame} if for any $G$-admissible $\tau : X \to
X$, every $G(F)$-invariant distribution on $X(F)$ is
$\tau$-invariant.

We call a symmetric pair $(G,H,\theta)$ \textbf{tame} if the
action of $H\times H$ on $G$ is tame.
\end{definition}

\begin{remark}
Evidently, any good tame symmetric pair is a GK pair.
\end{remark}

\begin{notation}
Let $V$ be an algebraic finite dimensional representation over $F$
of a reductive group $G$. Denote $Q(V):=V/V^G$. Since $G$ is
reductive, there is a canonical embedding $Q(V) \hookrightarrow
V$.
\end{notation}

\begin{notation}
Let $(G,H,\theta)$ be a symmetric pair. We denote by
$\mathcal{N}_{G,H}$ the subset of all the nilpotent elements in
$Q(\gd)$. Denote $R_{G,H}:=Q(\gd) - \mathcal{N}_{G,H}$.
\end{notation}
Note that our notion of $R_{G,H}$ coincides with the notion
$R(\gd)$ used in \cite{AG2}, 
Notation 2.3.10. This follows from
Lemma 7.1.11 in \cite{AG2}.

\begin{definition}
We call a symmetric pair $(G,H,\theta)$ \textbf{weakly linearly
tame} if for any $H$-admissible transformation $\tau$ of $\gd$
such that every $H(F)$-invariant distribution on $R_{G,H}$ is also
$\tau$-invariant,
we have\\
(*) every $H(F)$-invariant distribution on $Q(\gd)$ is also
$\tau$-invariant.
\end{definition}

\begin{theorem} \label{LinDes}
Let $(G,H,\theta)$ be a symmetric pair. Suppose that all its
descendants (including itself)  are weakly linearly tame. Then
$(G,H,\theta)$ is tame.
\end{theorem}
For proof see Theorem 7.3.3 in \cite{AG2}.

Now we would like to formulate a criterion for being weakly linearly
tame. For it we
will need the following lemma and notation.

\begin{lemma}
Let $(G,H,\theta)$ be a symmetric pair. Then any nilpotent element
$x \in \gd$ can be extended to an $sl_2$ triple $(x,d(x),x_-)$
such that $d(x) \in \h$ and $x_- \in \gd$.
\end{lemma}
For proof see e.g. \cite{AG2}, Lemma 7.1.11.

\begin{notation}
We will use the notation $d(x)$ from the last lemma in the future.
It is not uniquely defined but whenever we will use this notation
nothing will depend on its choice.
\end{notation}

\begin{proposition} \label{SpecCrit}
Let $(G,H,\theta)$ be a symmetric pair. Suppose that for any
nilpotent $x \in \gd$ we have
$$\tr(ad(d(x))|_{\h_x}) < \dim Q(\gd) .$$ 
Then the pair $(G,H,\theta)$ is weakly linearly tame.
\end{proposition}
This proposition follows from \cite{AG2} (Propositions 7.3.7 and
7.3.5).

\subsubsection{Regular symmetric pairs}
\begin{definition}
Let $(G,H,\theta)$ be a symmetric pair. We call an element $g \in
G(F)$ \textbf{admissible} if\\
(i) $Ad(g)$ commutes with $\theta$ (or, equivalently, $s(g)\in Z(G)$) and \\
(ii) $Ad(g)|_{\g^{\sigma}}$ is $H$-admissible.
\end{definition}

\begin{definition}
We call a symmetric pair $(G,H,\theta)$ \textbf{regular} if for
any admissible $g \in G(F)$ such that every $H(F)$-invariant
distribution on $R_{G,H}$ is also $Ad(g)$-invariant,
we have\\
(*) every $H(F)$-invariant distribution on $Q(\gd)$ is also
$Ad(g)$-invariant.
\end{definition}

The following two propositions are evident.
\begin{proposition} \label{TrivReg}
Let $(G,H,\theta)$ be symmetric pair. Suppose that any $g \in
G(F)$ satisfying $\sigma(g)g \in Z(G(F))$ lies in $Z(G(F))H(F)$.
Then $(G,H,\theta)$ is regular. In particular if the normalizer of
$H(F)$ lies inside $Z(G(F))H(F)$ then $(G,H,\theta)$ is regular.
\end{proposition}

\begin{proposition} $ $\\
(i) Any weakly linearly tame pair is regular. \\
(ii) A product of regular pairs is regular (see \cite{AG2},
Proposition 7.4.4).
\end{proposition}

In section \ref{GradRepDef} we will introduce terminology that will
help to verify the condition of Proposition \ref{SpecCrit}.

The importance of the notion of regular pair is demonstrated by
the following theorem.

\begin{theorem} \label{GoodHerRegGK}
Let $(G,H,\theta)$ be a good symmetric pair such that all its
descendants (including itself) are regular. Then it is a GK pair.
\end{theorem}
For proof see \cite{AG2}, Theorem 7.4.5.
\section{Main Results} \label{MainRes}
Here we formulate the main results of the paper and explain how
they follow from the rest of the paper.

\setcounter{lemma}{0}

\begin{definition}
A quadratic space is a linear space with a fixed non-degenerate
quadratic form.

Let $F'$ be an extension of $F$ and $V$ be a quadratic space over
it. We denote by $O(V)$ the canonical algebraic group such that
its $F$-points form the group of orthogonal transformations of
$V$.
\end{definition}

\begin{definition}
Let $D$ be a field with an involution $\tau$. A hermitian space
over $(D,\tau)$  is a linear space over $D$ with a fixed
non-degenerate hermitian form.

Suppose that $D$ is an extension of $F$ and $F \subset D^{\tau}$.
Let $V$ be a hermitian space over $(D,\tau)$. We denote by $U(V)$
the canonical algebraic group such that its $F$-points form the
group of unitary transformations of $V$.
\end{definition}

\begin{definition}
Let $G$ be a reductive group and $\eps \in G$ be an element of
order 2. We denote by $(G,G_{\eps})$ the symmetric pair defined by
the involution $x \mapsto \eps x \eps$.
\end{definition}

The following lemma is straightforward.
\begin{lemma}
Let $V$ be a quadratic space.\\
(i) Let $\eps \in GL(V)$ be an element of order 2. Then
$GL(V)_{\eps} \cong GL(V_1) \times GL(V_2)$ for some decomposition
$V=V_1 \oplus V_2$.\\
(ii) Let $\eps \in O(V)$ be an element of order 2. Then
$O(V)_{\eps} \cong O(V_1) \times O(V_2)$ for some orthogonal
decomposition $V=V_1 \oplus V_2$.\\
(iii) Let $V$ be a hermitian space.\\
 Let $\eps \in U(V)$ be an element of order 2. Then
$U(V)_{\eps} \cong U(V_1) \times U(V_2)$ for some orthogonal
decomposition $V=V_1 \oplus V_2$.
\end{lemma}

\begin{theorem} 
Let $V$ be a quadratic space over $F$. Then all the descendants of
the pair $(O(V),O(V)_{\eps})$ are regular.
\end{theorem}
\begin{proof}
By Theorem \ref{CompDesO_OtO} below, the descendants of the pair
$(O(V),O(V)_{\eps})$ are products of pairs of the types\\
(i) $(GL(W),O(W))$ for some quadratic space $W$ over some field
$F'$ that extends $F$\\
(ii) $(U(W_E) , O(W))$ for some quadratic space $W$ over some
field $F'$ that extends $F$, and some quadratic extension $E$ of
$F'$. Here, $W_E:=W \otimes _{F'}E$ is the extension of scalars
with the corresponding hermitian structure.\\
(iii) $(O(W),O(W)_{\eps})$ for some quadratic space $W$ over $F$.

The pair (i) is regular by Theorem \ref{GL_O} below. The pair (ii)
is regular by subsection \ref{U_reg} below. The pair (iii) is
regular by subsection \ref{O_OtO} below.
\end{proof}
\begin{corollary} 
Suppose that $F=\C$ and Let $V$ be a quadratic space over it. Then
the pair $(O(V),O(V)_{\eps})$ satisfies GP1.
\end{corollary}
\begin{proof}
This pair is good by Proposition \ref{GoodCrit} and all its
descendants are regular. Hence by Theorem \ref{GoodHerRegGK} it is
a GK pair. Therefore by Theorem \ref{DistCrit} it satisfies GP2.
Now, by Proposition \ref{GP2GP1}, it satisfies GP1.
\end{proof}
\begin{theorem}
Let $D/F$ be a quadratic extension and $\tau \in Gal(D/F)$ be the
non-trivial element. Let $V$ be a hermitian space over $(D,\tau)$.
Then all the descendants of the pair $(U(V),U(V)_{\eps})$ are
regular.
\end{theorem}
\begin{proof}
By theorem \ref{CompDesU_UtU} below, the descendants of the pair
$(U(V),U(V)_{\eps})$ are products of pairs of the types\\
(a) $(G \times G, \Delta G)$ for some reductive group $G$.\\
(b) $(GL(W),U(W))$ for some hermitian space $W$ over some
extension $(D',\tau')$ of $(D,\tau)$\\
(c) $(G_{E/F},G)$ for some reductive group $G$ and some quadratic
extension $E/F$.\\
(d) $(GL(W),GL(W)_{\eps})$ where $W$ is a linear space over $D$
and $\eps \in GL(W)$ is an element of order $\leq 2$.\\
(e) $(U(W),U(W)_{\eps})$ where $W$ is a hermitian space over
$(D,\tau)$.

The pairs (a) and (c) are regular by Theorem \ref{2RegPairs}
below. The pairs (b) and (e) are regular by subsection \ref{U_reg}
below. The pair (d) is regular by Theorem \ref{RJR} below.
\end{proof}
\begin{theorem} \label{Main_GL_O}
Let $V$ be a quadratic space over $F$. Then all the descendants of
the pair $(GL(V),O(V))$ are weakly linearly tame. In particular,
this pair is tame.
\end{theorem}
\begin{proof}
By Theorem \ref{CompDesGL_O} below, the descendants of the pair
$(GL(V),O(V))$ are products of pairs of the type $(GL(W),O(W))$
for some quadratic space $W$ over some field $F'$ that extends
$F$. By Theorem \ref{GL_O} below, these pairs are weakly linearly
tame. Now, the pair $(GL(V),O(V))$ is tame by Theorem
\ref{LinDes}.
\end{proof}
\begin{corollary} 
Suppose that $F=\C$ and Let $V$ be a quadratic space over it. Then
the pair $(GL(V),O(V))$ is GP1.
\end{corollary}

\begin{theorem} 
Let $D/F$ be a quadratic extension and $\tau \in Gal(D/F)$ be the
non-trivial element. Let $V$ be a hermitian space over $(D,\tau)$.
Then all the descendants of the pair $(GL(V),U(V))$ are weakly
linearly tame. In particular, this pair is tame.
\end{theorem}
\begin{proof}
By Theorem \ref{CompDesGL_U} below, all the descendants of the
pair
$(GL(V),U(V))$ are products of pairs of the types\\
(i) $(GL(W) \times GL(W), \Delta GL(W))$ for some linear space $W$
over some field $D'$ that extends $D$\\
(ii)  $(GL(W),U(W))$ for some hermitian space $W$ over some
$(D',\tau')$ that extends $(D,\tau)$.

The pair (i) is weakly linearly tame by Theorem \ref{2RegPairs}
below and the pair (ii) is weakly linearly tame by subsection
\ref{U_reg} below. Now, the pair $(GL(V),U(V))$ is tame by Theorem
\ref{LinDes}.
\end{proof}
\begin{theorem} 
Let $V$ be a quadratic space over $F$. Let $D/F$ be a quadratic
extension and $\tau \in Gal(D/F)$ be the non-trivial element. Let
$V_D:=V \otimes_F D$ be its extension of scalars with the
corresponding hermitian structure. Then all the descendants of the
pair $(U(V_D),O(V))$ are weakly linearly tame. In particular, this
pair is tame.
\end{theorem}
\begin{proof}

By Theorem \ref{CompDesU_O} below, all the descendants of the pair
$(U(V_D),O(V))$ are products
of pairs of the types\\
(i) $(GL(W),O(W))$ for some quadratic space $W$ over some field
$F'$ that extends $F$.\\
(ii) $(U(W_{D'}),O(W))$ for some extension $(D',\tau')$ of
$(D,\tau)$ and some quadratic space $W$ over $D'^{\tau'}$.

The pair (i) is weakly linearly tame by Theorem \ref{GL_O} below
and the pair (ii) is weakly linearly tame by subsection
\ref{U_reg} below. Now, the pair $(GL(V),U(V))$ is tame by Theorem
\ref{LinDes}.
\end{proof}

\section{$\Z/2\Z$ graded representations of $sl_2$ and their
defects} \label{GradRepDef}

In this section we will introduce terminology that will help to
verify the condition of Proposition \ref{SpecCrit}.

\subsection{Graded representations of $sl_2$}

\begin{definition}
We fix standard basis $e,h,f$ of $sl_2(F)$. We fix a grading on
$sl_2(F)$ given by $h \in sl_2(F)_0$ and $e,f \in sl_2(F)_1$. A
\textbf{graded representation of $sl_2$} is a representation of
$sl_2$ on a graded vector space $V=V_0 \oplus V_1$ such that
$sl_2(F)_i(V_j) \subset V_{i+j}$ where $i,j \in \Z/2\Z$.
\end{definition}

The following lemma is standard.
\begin{lemma}$ $\\
(i) Every graded representation of $sl_2$ which is
irreducible as a graded representation is irreducible just as a representation.\\
(ii) Every irreducible representation $V$ of $sl_2$ admits exactly
two gradings. In one highest weight vector lies in $V_0$ and in
the other in $V_1$.
\end{lemma}

\begin{definition}
We denote by $V_{\lambda}^{w}$ the irreducible graded
representation of $sl_2$ with highest weight $\lambda$ and highest
weight vector of parity $p$ where $w = (-1)^p$.
\end{definition}

The following lemma is straightforward.

\begin{lemma} \label{GradedTensors}
\setcounter{equation}{0}
\begin{align}
& (V_{\lambda}^w)^* = V_{\lambda}^{w(-1)^{\lambda}}\\
& V_{\lambda_1}^{w_1} \otimes V_{\lambda_2}^{w_2} =
\bigoplus_{i=0}^{\min(\lambda_1, \lambda_2)}
V_{\lambda_1+\lambda_2 - 2i}^{w_1w_2(-1)^i}\\
& \Lambda^2(V_{\lambda}^w) = \bigoplus_{i=0}^{\lfloor
\frac{\lambda-1}{2} \rfloor} V_{2\lambda-4i-2}^{-1}.
\end{align}
\end{lemma}

\subsection{Defects}
\begin{definition}
Let $\pi$ be a graded representation of $sl_2$. We define the
\textbf{defect} of $\pi$ to be 
$$\de(\pi)=\tr(h|_{(\pi^e)_0})-\dim(\pi_1)$$
\end{definition}
The following lemma is straightforward

\begin{lemma} \label{Defects}

\begin{align}
&\de(\pi \oplus \tau)=\de(\pi) + \de(\tau)\\
&\de(V_{\lambda}^w) =\frac{1}{2}(\lambda w + w( \frac{1
+(-1)^{\lambda}}{2}) -1)= \frac{1}{2} \left\{%
\begin{array}{ll}
   \lambda w + w -1 & \lambda \text{ is even} \\
    \lambda w -1 & \lambda \text{ is odd}
\end{array}%
\right.
\end{align}

\end{lemma}
\begin{definition}
Let $\g$ be a $(\Z/2\Z)$ graded Lie algebra. We say that $\g$
\textbf{is of negative defect} if for any graded homomorphism
$\pi: sl_2 \to \g$, the defect of $\g$ with respect to the adjoint
action of $sl_2$ is negative. 

We say that  $\g$
\textbf{is of negative normalized defect} if the semi-simple part of $\g$ (i.e. the quotient of $\g$ by its center) 
is of negative defect.
\end{definition}
\begin{remark}
Clearly,  $\g$
is of negative normalized defect if and only if for any graded homomorphism
$\pi: sl_2 \to \g$, the defect of $\g$ with respect to the adjoint
action of $sl_2$ is less than the dimension of the odd part of the center of  $\g$.
\end{remark}
\begin{definition}
We say that a symmetric pair $(G,H,\theta)$ \textbf{is of negative 
normalized defect} if the Lie algebra $\g$ with the grading defined by
$\theta$ is of negative 
normalized defect.
\end{definition}

\begin{lemma}
Let $(G,H,\theta)$  be a symmetric pair. Assume that $\g$ is semi-simple. Then
$Q(\gd)=\gd.$
\end{lemma}
\begin{proof}$ $\\
Assume the contrary: there exist $0 \neq x \in \gd$ such that
$Hx=x$. Then $\dim (CN_{Hx,x}^{\gd}) = \dim \gd$, hence
$CN_{Hx,x}^{\gd} = \gd$. On the other hand, $CN_{Hx,x}^{\gd} \cong [\h, x]^\bot=(\gd)^x$ (here $(\cdot)^\bot$ means the ortogonal compliment w.r.t. the kiling form). Therefore $\gd=(\gd)^x$ and hence $x$
lies in the center of $\g$, which is impossible.
\end{proof}

Proposition \ref{SpecCrit} can be rewritten now in the following
form
\begin{theorem}
A symmetric pair of negative 
normalized defect is weakly linearly tame.
\end{theorem}

Evidently, a product of pairs of negative normalized defect is again of
negative normalized defect.

The following lemma is straightforward.
\begin{lemma} \label{NegDefAlg}
Let $(G,H,\theta)$ be a symmetric pair. Let $F'$ be any field
extending $F$. Let $(G_{F'},H_{F'},\theta)$ be the extension of
$(G,H,\theta)$ to $F'$. Suppose that it is of negative normalized defect (as
a pair over $F'$) . Then $(G,H,\theta)$ and
$(G_{F'/F},H_{F'/F},\theta)$ are of negative normalized defect (as pairs over
$F$).
\end{lemma}

In \cite{AG2} we proved the following (easy) proposition (see
\cite{AG2}, Lemma 7.6.6).
\begin{proposition} \label{pi_pibarAG2}
Let $\pi$ be a representation of $sl_2$. Then
$\tr(h|_{(\pi^e)})<\dim(\pi)$.
\end{proposition}
We would like to reformulate it in terms of defect. For this we
will need the following notation.
\begin{notation} $ $\\
(i) Let $\pi$ be a representation of $sl_2$. We denote by $\opi$
the representation of $sl_2$ on the same space defined by
$\opi(e):=-\pi(e)$, $\opi(f):=-\pi(f)$ and $\opi(h):=\pi(h)$. \\
(ii) We define grading on $\pi \oplus \opi$ by the involution $s(v
\oplus w):=  w \oplus v$.
\end{notation}
Proposition \ref{pi_pibarAG2} can be reformulated in the following
way.
\begin{proposition} \label{pi_pibar}
Let $\pi$ be a representation of $sl_2$. Then $\de(\pi \oplus
\opi)<0$.
\end{proposition}
In \cite{AG2} we also deduced from this proposition the following
theorem (see \cite{AG2}, 7.6.2).
\begin{theorem} \label{2RegPairs}
For any reductive group $G$, the pairs $(G \times G,\Delta G)$ and
$(G_{E/F},G)$ are of negative normalized defect and hence weakly linearly
tame. Here $\Delta G$ is the diagonal in $G \times G$.
\end{theorem}
In 
\cite[\S\S 2.7]{AG2} we 
proved the following theorem.
\begin{theorem} \label{RJR}
The pair $(GL(V \oplus V),GL(V) \times GL(V))$ is of negative
normalized defect and hence regular.
\end{theorem}
Note that in the case $dimV \neq dimW$ the pair $(GL(V \oplus
W),GL(V) \times GL(W))$ is obviously regular by Proposition
\ref{TrivReg}.

\section{Proof of regularity and tameness} \label{SecReg}

\subsection{The pair $(GL(V),O(V))$}  $ $\\\\
In this subsection we prove that the pair $(GL(V),O(V))$ is weakly
linearly tame.  
For $\dim V \leq 1$  it is obvious.
Hence it is enough to prove the following
theorem.

\begin{theorem} \label{GL_O}
Let $V$ be a quadratic space 
of dimension at least $2$. Then the pair $(GL(V),O(V))$ has
negative normalized defect.
\end{theorem}
We will need the following notation.

\begin{notation}
Let $\pi$ be a representation of $sl_2$. We define grading on $\pi
\otimes \opi$ by the involution $s(v \otimes w):= - w \otimes v$.
\end{notation}

Theorem \ref{GL_O} immediately follows from the following one.

\begin{theorem}
Let $\pi$ be a representation of $sl_2$ of dimension  at least 2. Then $def(\pi \otimes \opi)<-1.$
\end{theorem}

This theorem in turn follows from the following lemma.

\begin{lemma}
Let $V_{\lambda}$ and $V_{\mu}$ be irreducible
representations of $sl_2$. Then\\
(i) $def(V_{\lambda} \otimes \overline{V_{\lambda}}) = -(\lambda +
1)(\frac{\lambda}{2}+1)$.\\
(ii) $def(V_{\lambda} \otimes \overline{V_{\mu}} \oplus V_{\mu}
\otimes \overline{V_{\lambda}})<0$.
\end{lemma}
\begin{proof} $ $\\
(i) Follows from the fact that $V_{\lambda} \otimes
\overline{V_{\lambda}} = \bigoplus_{i=0}^{\lambda} V_{2\lambda -
2i}^{-1}$ and from Lemma \ref{Defects}.\\
(ii) Follows from Proposition \ref{pi_pibar}.
\end{proof}

\subsection{The pair $(O(V_1 \oplus V_2), O(V_1) \times O(V_2))$} \label{O_OtO}
 $ $\\\\
In this subsection prove that the pair $(O(V_1 \oplus V_2), O(V_1)
\times O(V_2))$ is regular. For that it is enough to prove the
following theorem.

\begin{theorem}
Let $V_1$ and $V_2$ be quadratic spaces. Assume $\dim V_1 = \dim
V_2$. Then the pair $(O(V_1 \oplus V_2), O(V_1) \times O(V_2))$
has negative normalized defect.
\end{theorem}

This theorem immediately follows from the following one.

\begin{theorem} \label{Lambda2neg}
Let $\pi$ be a (non-zero) graded representation of $sl_2$ such
that $\dim \pi_0 = \dim \pi_1$ and $\pi \simeq \pi^*$. Then
$\Lambda^2(\pi)$ has negative defect.
\end{theorem}

For this theorem we will need the following lemma.

\begin{lemma}
Let $V_{\lambda_1}^{w_1}$ and $V_{\lambda_2}^{w_2}$ be irreducible
graded
representations of $sl_2$. Then\\
(i) $def(V_{\lambda_1}^{w_1} \otimes V_{\lambda_2}^{w_2}) =$
$$= -\frac{1}{2} \left\{%
\begin{array}{lll}
    \min(\lambda_1,\lambda_2)+1 - \frac{w_1w_2}{2}(\lambda_1+\lambda_2+1+(-1)^{\min(\lambda_1,\lambda_2)}(|\lambda_1-\lambda_2|-1)), & \lambda_1 \neq \lambda_2 & \mod 2; \\
    \min(\lambda_1,\lambda_2)+1 - w_1w_2(\max(\lambda_1,\lambda_2)+1), & \lambda_1 \equiv \lambda_2 \equiv 0 & \mod 2; \\
    \min(\lambda_1,\lambda_2)+1 - w_1w_2(\min(\lambda_1,\lambda_2)+1), & \lambda_1 \equiv \lambda_2 \equiv 1 & \mod 2;
\end{array}%
\right.$$
(ii) $def(\Lambda^2(V_{\lambda_1}^{w_1}))= -\frac{\lambda^2}{4} -
\frac{\lambda}{2} - \frac{1 + (-1)^{\lambda+1}}{8}$
\end{lemma}
\begin{proof}
This lemma follows by straightforward computations from Lemmas
\ref{GradedTensors} and \ref{Defects}.
\end{proof}

\begin{proof}[Proof of Theorem \ref{Lambda2neg}]
Since $\pi \simeq \pi^*$, $\pi$ can be decomposed to a direct sum
of irreducible graded representations in the following way
$$ \pi = (\bigoplus_{i=1}^l V_{\lambda_i}^{1})  \oplus
(\bigoplus_{j=1}^m V_{\mu_j}^{-1}) \oplus (\bigoplus_{k=1}^n
V_{\nu_k}^{1} \oplus V_{\nu_k}^{-1}).$$

Here, all $\lambda_i$ and $\mu_j$ are even and $\nu_k$ are odd.
Since $\dim \pi_0 = \dim \pi_1$, $l=m$.

By the last lemma, $def(V_{\lambda_i}^1 \otimes (V_{\nu_k}^{1}
\oplus V_{\nu_k}^{-1})) = -(\min(\lambda_1,\lambda_2)+1)<0$.
Similarly, $def(V_{\mu_j}^{-1} \otimes (V_{\nu_k}^{1} \oplus
V_{\nu_k}^{-1}))<0$. Also, $def((V_{\nu_{k_1}}^{1} \oplus
V_{\nu_{k_1}}^{-1}) \otimes (V_{\nu_{k_2}}^{1} \oplus
V_{\nu_{k_2}}^{-1}))<0$ and $def(\Lambda^2(V_{\lambda}^w)) \leq 0$
for all $\lambda$ and $w$.

Hence if $l=0$ we are done. Otherwise we can assume $n=0$. Now,

\begin{multline*}
def(\Lambda^2(\pi)) \leq \sum_{1 \leq i<j \leq
l}|\lambda_i-\lambda_j| + \sum _{1 \leq i<j \leq l}|\mu_i-\mu_j| -
\sum_{1 \leq i,j \leq l}(\lambda_i+\mu_j+2)<\\
< \sum_{1 \leq i<j \leq l}(\lambda_i+\lambda_j) + \sum _{1 \leq
i<j \leq l}(\mu_i+\mu_j) - \sum_{1 \leq i,j \leq
l}(\lambda_i+\mu_j)=-\sum_{i=1}^l (\lambda_i+\mu_i) \leq 0.
\end{multline*}

\end{proof}

\subsection{The pairs $(GL(V),U(V)), \quad (U(V_1 \oplus V_2), U(V_1) \times U(V_2)) \text{ and } (U(V_D),O(V))$
} \label{U_reg}
$ $\\

In this subsection prove that the pairs $(GL(V),U(V))$ and
$(U(V_D),O(V))$ are weakly linearly tame and the pair $(U(V_1
\oplus V_2), U(V_1) \times U(V_2))$ is regular.\\

Let $V$ be a hermitian space. Note that $(GL(V),U(V))$ is a form
of $(GL(W) \times GL(W), \Delta GL(W))$ for some $W$ and $(U(V
\oplus V), U(V) \times U(V))$ is a form of $(GL(W \oplus W),
GL(W)) \times GL(W))$ for some $W$.

Also, for any quadratic space $V$ of dimension at least 2 and any quadratic extension
$D/F$, the pair $(U(V_D),O(V))$ is a form of $(GL(W),O(W))$ for
some quadratic space $W$.

Hence by Lemma \ref{NegDefAlg} and Theorems \ref{2RegPairs},
\ref{RJR} and \ref{GL_O} those 3 pairs are of negative normalized defect and
hence are weakly linearly tame.
If $\dim V \leq 1$ then the pair $(U(V_D),O(V))$ is obviously linearly tame.

If $V_1$ and $V_2$ are non-isomorphic hermitian spaces then
$(U(V_1 \oplus V_2), U(V_1) \times U(V_2))$ is regular by
Proposition \ref{TrivReg}.

\section{Computation of descendants} \label{CompDes}
In this section we compute the descendants of the pairs we
discussed before. For this we use a technique of computing
centralizers of semisimple elements of orthogonal and unitary
groups, which is described in \cite{SpSt}. The proofs in this
section are rather straightforward but technically involved. The
most important things in this section are the formulations of the
main theorems: Theorems \ref{CompDesGL_O}, \ref{CompDesGL_U},
\ref{CompDesU_O}, \ref{CompDesO_OtO}, \ref{CompDesU_UtU}. Those
theorems are summarized graphically in subsection \ref{SecDiag}.

\subsection{Preliminaries and notation for orthogonal and unitary groups}

\subsubsection{Orthogonal group}
\begin{notation}
Let $V$ be a linear space over $F$. Let $x\in GL(V)$ be a
semi-simple element and let $Q = \sum_{i=0}^{n} a_i \xi^i \in
F[\xi]$ (where $a_n \neq 0$) be an irreducible polynomial.
\itemize{
\item Denote $F_Q:= F[\xi]/Q$
\item Denote $inv(Q):= \sum_{i=0}^n a_{n-i}\xi^i$
\item Denote $V_{Q,x}^{0}:=Ker(Q(x)) \text{ and }
V_{Q,x}^{1}:=Ker(inv(Q)(x))$\\
We define an $F_Q$-linear space structure on $V_{Q,x}^{i}$ by
letting $\xi$ act on $V_{Q,x}^{0}$ by $x$ and on $V_{Q,x}^{1}$ by
$x^{-1}$. We will consider $V_{Q,x}^{i}$ as linear spaces over
$F_Q$.

\item In case $Q$ is proportional to $inv(Q)$ we define an involution $\mu$ on $F_Q$
by $\mu(P(\xi)):=P(\xi^{-1})$ .
\item For a linear space $W$ over $F_Q$ we can consider its dual
space $W^*$ over $F_Q$ and the dual space of $W$ over $F$ which we
denote by $W_F^*$. The space $W_F^*$ has a canonical structure of
a linear space over $F_Q$. The spaces $W_F^*$ and $W^*$ can be
identified as linear spaces over $F_Q$. For this identification
one has to choose an $F$-linear functional $\lambda:F_Q \to F$. We
will fix such functional $\lambda$ such that
$\lambda(\mu(d))=\lambda(d)$ if $\mu$ is defined.\\ From now on we
will identify $W_F^*$ and $W^*$. }
\end{notation}

The following two lemmas are straightforward.
\begin{lemma}
Let $V$ be a quadratic space over $F$. Let $x\in GL(V)$
 and let $P,Q \in F[\xi]$ be irreducible
polynomials. Suppose that either \\
(i) $x=x^t$ and $P$ is not proportional to $Q$ or\\
(ii) $x \in O(V)$ and $P$ is not proportional to $inv(Q)$\\
Then $Ker(Q(x))$ is orthogonal to $Ker(P(x))$.
\end{lemma}

\begin{lemma}
Let $(V,B)$ be a quadratic space over $F$. Let $x\in GL(V)$ be a
semi-simple element and let $Q \in F[\xi]$ be an
irreducible polynomial. Then\\
(i) If $x=x^t$ then $B$ defines an $F_Q$-linear isomorphism
$V_{Q,x}^i \cong
(V_{Q,x}^i)^*$.\\
(ii) If $x\in O(V)$ then $B$ defines an $F_Q$-linear isomorphism
$V_{Q,x}^i \cong (V_{Q,x}^{1-i})^*$.
\end{lemma}

\subsubsection{Unitary group}
$ $\\
From now and till the end of the paper we fix a quadratic
extension $D$ of $F$ and denote by $\tau$ the involution that
fixes $F$.

\begin{notation}
Let $V$ be a hermitian space over $(D,\tau)$. Let $x\in GL(V)$ be
a semi-simple element and let $Q = \sum_{i=0}^{n} a_i \xi^i \in
D[\xi]$ (where $a_n \neq 0$) be an irreducible polynomial.
\itemize{
\item Denote $D_Q:= D[\xi]/Q$\\
\item Denote $$inv(Q):= \sum_{i=0}^n a_{n-i}\xi^i, \quad
Q^*:=\tau(inv(Q))$$
\item Denote $$V_{Q,x}^{00}:=Ker(Q(x)), \quad
V_{Q,x}^{01}:=Ker(Q^*(x)) , \quad V_{Q,x}^{10}:=Ker(inv(Q)(x)) ,
\quad V_{Q,x}^{11}:=Ker(\tau(Q)(x)).$$
We twist the action of $D$ on $V_{Q,x}^{i1}$ by $\tau$. We define
$D_Q$-linear space structure on $V_{Q,x}^{ij}$ by letting $\xi$
act on $V_{Q,x}^{0j}$ by $x$ and on $V_{Q,x}^{1j}$ by $x^{-1}$. We
will consider $V_{Q,x}^{ij}$ as linear spaces over $D_Q$.

\item If $Q$ is proportional to $Q^*$ we define an involution $\mu^{01}$ on $D_Q$ by
$\mu^{01}(P(\xi)):=\tau(P)(\xi^{-1})$.\\
If $Q$ is proportional to $inv(Q)$ we define an involution
$\mu^{10}$ on $D_Q$ by
$\mu^{10}(P(\xi)):=P(\xi^{-1})$.\\
If $Q$ is proportional to $\tau(Q)$ we define an involution
$\mu^{11}$ on $D_Q$ by $\mu^{11}(P(\xi)):=\tau(P)(\xi)$.

\item For a linear space $W$ over $D_Q$ we can consider its dual
space $W^*$ over $D_Q$ and the dual space of $W$ over $D$ which we
denote by $W_D^*$. The space $W_D^*$ has a canonical structure of
a linear space over $D_Q$. The spaces $W_D^*$ and $W^*$ can be
identified as linear spaces over $D_Q$. For this identification
one has to choose a $D$-linear functional $\lambda:D_Q \to D$. We
will fix such functional $\lambda$ such that
$$ \lambda(\mu^{ij}(d))=\tau^{j}(\lambda(d)) \quad \text{ if } \mu^{ij}
\text{ is defined}.$$
%
From now on we will identify $W_D^*$ and $W^*$. }
\end{notation}

The following two lemmas are straightforward.
\begin{lemma}
Let $V$ be a hermitian space over $(D,\tau)$. Let $x\in GL(V)$ and
let $P,Q \in D[\xi]$ be irreducible
polynomials. Suppose that either \\
(i) $x=x^*$ and $P$ is not proportional to $\tau(Q)$ or\\
(ii) $x \in U(V)$ and $P$ is not proportional to $Q^*$\\
Then $Ker(Q(x))$ is orthogonal to $Ker(P(x))$.
\end{lemma}

\begin{lemma}
Let $(V,B)$ be a hermitian space over $(D,\tau)$. Let $x\in GL(V)$
be a semi-simple element and let $Q \in D[\xi]$ be an
irreducible polynomial. Then\\
(i) If $x=x^*$ then $B$ defines a $D_Q$-linear isomorphism
$V_{Q,x}^{ij} \cong
(V_{Q,x}^{1-i,1-j})^*$.\\
(ii) If $x\in U(V)$ then $B$ defines a $D_Q$-linear isomorphism
$V_{Q,x}^{ij} \cong (V_{Q,x}^{i,1-j})^*$.
\end{lemma}
\subsection{The pair $(GL(V),O(V))$}
\begin{theorem} \label{CompDesGL_O}
Let $V$ be a quadratic space over $F$. Then all the descendants of
the pair $(GL(V),O(V))$ are products of pairs of the type
$(GL(W),O(W))$ for some quadratic space $W$ over some field $F'$
that extends $F$.
\end{theorem}
\begin{proof}
Note that in this case the anti-involution $\sigma$ is given by
$\sigma(x)=x^t$. Let $x\in GL(V)^{\sigma} $ be a semi-simple
element. Let $P$ be the minimal polynomial of $x$. We will now
discuss a special case and then deduce the general case from
it.\\\\
Case 1. $P$ is irreducible over $F$.\\
Clearly $GL(V)_x \cong GL(V_{P,x}^0)$. The isomorphism $V_{P,x}^0
\cong (V_{P,x}^0)^*$ gives a quadratic structure on $V_{P,x}^0$.
Now
 $O(V)_x \cong O(V_{P,x}^0)$.\\\\
Case 2. General case\\
Let $P=\prod_{i \in I}P_i$ be the decomposition of $P$ to
irreducible polynomials. Clearly $V = \bigoplus V_{P_i,x}^0$ and
$V_{P_i,x}^0$ are orthogonal to each other. Hence the pair
$(GL(V)_x,O(V)_x)$
 is a product of pairs from Case 1.
\end{proof}

\subsection{The pair $(GL(V),U(V))$}

\begin{theorem} \label{CompDesGL_U}
Let $(V,B)$ be a hermitian space over $(D,\tau)$. Then all the
descendants
of the pair $(GL(V),U(V))$ are products of pairs of the types\\
(i) $(GL(W) \times GL(W), \Delta GL(W))$ for some linear space $W$
over some field $D'$ that extends $D$\\
(ii)  $(GL(W),U(W))$ for some hermitian space $W$ over some
$(D',\tau')$ that extends $(D,\tau)$.
\end{theorem}
\begin{proof}
Note that in this case the anti-involution $\sigma$ is given by
$\sigma(x)=x^*$. Let $x\in GL(V)^{\sigma} $ be a semi-simple
element. Let $P$ be the minimal polynomial of $x$. Note that
$\tau(P)$ is proportional to $P$. We will now discuss 2 special
cases and then deduce the general case from
them.\\\\
Case 1. $P=Q\tau(Q)$ where $Q$ is irreducible over $D$.\\
Clearly $GL(V)_x \cong GL(V_{Q,x}^{00}) \times GL(V_{Q,x}^{11})$.
Recall that $B$ gives a non-degenerate pairing between
$V_{Q,x}^{00}$ and $V_{Q,x}^{11}$, and the spaces $V_{Q,x}^{ii}$
are isotropic. Therefore
$$GL(V_{Q,x}^{00})\cong GL(V_{Q,x}^{11}), \quad GL(V)_x \cong
GL(V_{Q,x}^{00})^2 \text{ and } U(V)_x \cong \Delta GL(V_{Q,x}^{00}) < GL(V_{Q,x}^{00})^2. $$\\
Case 2. $P$ is irreducible over $D$.\\
Clearly $GL(V)_x \cong GL(V_{P,x}^{00})$ and $V_{P,x}^{00}$ is
identical to $V_{P,x}^{11}$ as $F$-linear spaces but the actions
of $D_P$ differ by a twist by $\mu^{11}$. Hence the isomorphism
$V_{P,x}^{00} \cong (V_{P,x}^{11})^*$ gives a hermitian structure
on $V_{P,x}^{00}$ over $(D_P,\mu^{11})$. Now
$U(V)_x \cong U(V_{P,x}^{00}) < GL(V_{P,x}^{00})$.\\\\
Case 3. General case\\
Let $P=\prod_{i \in I}P_i$ be the decomposition of $P$ to
irreducible polynomials. Then $\tau(P_i)$ is proportional to
$P_{s(i)}$ where $s$ is some permutation of $I$ of order $\leq$ 2.
Let $I = \bigsqcup I_{\alpha}$ be the decomposition of $I$ to
orbits of $s$. Denote $V_{\alpha}:=Ker (\prod _{i \in \alpha}
P_i(x))$. Clearly $V = \bigoplus V_{\alpha}$ and $V_{\alpha}$ are
orthogonal to each other. Hence the pair $(GL(V)_x,U(V)_x)$ is a
product of pairs from the first 2 cases.
\end{proof}

\subsection{The pair $(U(V_D),O(V))$}
\begin{theorem} \label{CompDesU_O}
Let $(V,B)$ be a quadratic space over $F$. Let $V_D:=V \otimes_F
D$ be its extension of scalars with the corresponding hermitian
structure.

Then all the descendants of the pair $(U(V_D),O(V))$ are products
of pairs of the types\\
(i) $(GL(W),O(W))$ for some quadratic space $W$ over some field
$F'$ that extends $F$.\\
(ii) $(U(W_{D'}),O(W))$ for some extension $(D',\tau')$ of
$(D,\tau)$ and some quadratic space $W$ over $D'^{\tau'}$.
\end{theorem}

For the proof of this theorem we will need the following notation
and lemma.

\begin{notation}
Let $(V,B)$ be a quadratic space over $F$. The involution $\tau$
defines an involution $\widetilde{\tau}$ on $V_D$. The form $B$
defines a quadratic form $B_D$ on $V_D$ and a hermitian form
$B_D^{\tau}$ on $V_D$.
\end{notation}


\begin{lemma}
Let $(V,B)$ be a quadratic space over $F$. Let $P$ be an
irreducible polynomial. Let $x \in U(V_D)$ be a semi-simple
element such that $x = x^t$ (where $x^t$ is defined by $B_D$).
Then the involution $\widetilde{\tau}$ gives a $D_{P}$-linear
isomorphism $V_{P,x}^{ij} \cong V_{P,x}^{i,1-j}$.
\end{lemma}
\begin{proof}
We will show that $\widetilde{\tau}$ maps $V_{P,x}^{00}$ to
$V_{P,x}^{11}$, and the other cases are done similarly. Let $v \in
V_{P,x}^{ij}$. We have
$$ P^*(x)(\widetilde{\tau}(v))=
\widetilde{\tau}(inv(P)(x^*)(v))=
\widetilde{\tau}(inv(P)(x^{-1})(v))=\widetilde{\tau}(x^{-degP}P(x)(v))=0.$$
\end{proof}

\begin{proof}[Proof of Theorem \ref{CompDesU_O}]
Note that in this case the anti-involution $\sigma$ is given by
$\sigma(x)=x^t$. Let $x\in U(V_D)^{\sigma} $ be a semi-simple
element. Let $P$ be the minimal polynomial of $x$. Then $P$ is
proportional to $P^*$. We will now discuss 2 special cases and
then deduce the general case from
them.\\\\
Case 1. $P=QQ^*$ where $Q$ is irreducible over $D$.\\
Clearly $GL(V)_x \cong GL(V_{Q,x}^{00}) \times GL(V_{Q,x}^{01})$.
Recall that $B_D^{\tau}$ gives a non-degenerate pairing between
$V_{Q,x}^{00}$ and $V_{Q,x}^{01}$, and the spaces $V_{Q,x}^{0i}$
are isotropic. Therefore $$GL(V_{Q,x}^{00})\cong GL(V_{Q,x}^{01}),
\quad GL(V)_x \cong GL(V_{Q,x}^{00})^2, \quad  U(V)_x \cong \Delta
GL(V_{Q,x}^{00}) < GL(V_{Q,x}^{00})^2$$
Compose the isomorphism $V_{Q,x}^{00} \cong V_{Q,x}^{01}$ given by
$\widetilde{\tau}$ with the isomorphism $V_{Q,x}^{01} \cong
(V_{Q,x}^{00})^*$ given by $B_D^{\tau}$. This gives a quadratic
structure on $V_{Q,x}^{00}$. Now
$$O(V)_x \cong
\Delta O(V_{Q,x}^{00}) < \Delta GL(V_{Q,x}^{00}). $$\\
Case 2. $P$ is irreducible over $D$.\\
Clearly $GL(V)_x \cong GL(V_{P,x}^{00})$ and $V_{P,x}^{00}$ is
identical to $V_{P,x}^{01}$ as $F$-linear spaces but the actions
of $D_P$ on them differ by a twist by $\mu^{01}$. Hence the
isomorphism $V_{P,x}^{00} \cong (V_{P,x}^{01})^*$ given by
$B_D^{\tau}$ gives a hermitian structure on $V_{P,x}^{00}$ over
$(D_P,\mu^{01})$ and the isomorphism $V_{P,x}^{00} \cong
V_{P,x}^{01}$ given by $\widetilde{\tau}$ gives an antilinear
involution of $V_{P,x}^{00}$. Now
$$U(V)_x \cong U(V_{P,x}^{00}) < GL(V_{P,x}^{00}) \text{ and }O(V)_x \cong O(V_{P,x}^{00}) < U(V_{P,x}^{00})
.$$\\
Case 3. General case\\
Let $P=\prod_{i \in I}P_i$ be the decomposition of $P$ to
irreducible polynomials. Then $P_i^*$ is proportional to
$P_{s(i)}$ where $s$ is some permutation of $I$ of order $\leq$ 2.
Let $I = \bigsqcup I_{\alpha}$ be the decomposition of $I$ to
orbits of $s$. Denote $V_{\alpha}:=Ker (\prod _{i \in \alpha}
P_i(x))$. Clearly $V_D = \bigoplus V_{\alpha}$, $V_{\alpha}$ are
orthogonal to each other and each $V_{\alpha}$ is invariant with
respect to $\widetilde{\tau}$. Hence the pair $(GL(V)_x,U(V)_x)$
is a product of pairs from the first 2 cases.
\end{proof}

\subsection{The pair $(O(V_1 \oplus V_2), O(V_1) \times O(V_2))$}

\begin{theorem} \label{CompDesO_OtO}
Let $(V,B)$ be a quadratic space over $F$.\\
Let $\eps \in O(V)$ be an element of order 2. Then all the
descendants of the pair $(O(V),O(V)_{\eps})$ are products of pairs
of the types\\
(i) $(GL(W),O(W))$ for some quadratic space $W$ over some field
$F'$ that extends $F$\\
(ii) $(U(W_E) , O(W))$ for some quadratic space $W$ over some
field $F'$ that extends $F$, and some quadratic extension $E$ of
$F'$.\\
(iii) $(O(W),O(W)_{\eps})$ for some quadratic space $W$ over $F$.
\end{theorem}

For the proof of this theorem we will need the following
straightforward lemma.

\begin{lemma}
Let $(V,B)$ be a quadratic space over $F$.\\
Let $\eps \in O(V)$ be an element of order 2. Let $x \in O(V)$
such that $\eps x \eps = x^{-1}$. Let $Q$ be an irreducible
polynomial. Then $\eps$ gives an $F_Q$- linear isomorphism
$V_{Q,x}^i \cong V_{Q,x}^{1-i}$.
\end{lemma}

\begin{proof}[Proof of Theorem \ref{CompDesO_OtO}]
Note that the involution $\sigma$ on $O(V)$ is given by $x \mapsto
\eps x^{-1}\eps$. Let $x\in O(V)^{\sigma}$ be a semi-simple
element and let $P$ be its minimal polynomial.

Note that the minimal polynomial of $x^{-1}$ is $inv(P)$ and hence
$P$ is proportional to $inv(P)$. We will now discuss 3 special
cases and then deduce the general case from
them.\\\\
Case 1. $P=Qinv(Q)$, where $Q$ is an irreducible polynomial.\\
Note that $GL(V)_x \cong \prod _{i} GL(V_{Q,x}^{i})$.

Since $B$ defines a non-degenerate pairing $V_{Q,x}^{0} \cong
(V_{Q,x}^{1})^*$, and $V_{Q,x}^{i}$ are isotropic, we have
$$O(V)_x \cong \Delta GL(V_{Q,x}^{0}) < GL(V_{Q,x}^{0})^2.$$
Now, compose the isomorphism $V_{Q,x}^{i} \cong V_{Q,x}^{1-i}$
given by $\eps$ with the isomorphism $V_{Q,x}^{1-i} \cong
(V_{Q,x}^{i})^*$. This gives a quadratic structure on
$V_{Q,x}^{0}$. Clearly, $\eps$ gives an isomorphism $V_{Q,x}^{0}
\cong V_{Q,x}^{1}$ as quadratic spaces and hence
$$(O(V)_{\eps})_x \cong  \Delta O(V_{Q,x}^{0}) < \Delta
GL(V_{Q,x}^{0}).$$\\
Case 2. $P$ is irreducible and $x \neq x^{-1}$\\
In this case $GL(V)_x \cong GL(V_{P,x}^{0})$. Also, $V_{P,x}^{0}$
and $V_{P,x}^{1}$ are identical as $F$-vector spaces but the
action of $F_P$ on them differs by a twist by $\mu$. Therefore the
isomorphism $V_{P,x}^{0} \cong (V_{P,x}^{1})^*$ gives a hermitian
structure on $V_{P,x}^{0}$ over $(F_P,\mu)$ and $\eps$ gives an
$(F_P,\mu)$-antilinear automorphism of $V_{P,x}^{0}$. Now
$$O(V)_x \cong U(V_{P,x}^{0}).$$
Denote $W:= (V_{P,x}^{0})^{\eps}$. It is a linear space over
$(F_P)^{\mu}$. It has a quadratic structure. Now
$$(O(V)_{\eps})_x \cong O(W) < U(V_{P,x}^{0}).$$\\
Case 3. $P$ is irreducible and $x = x^{-1}$.\\
Again, $GL(V)_x \cong GL(V_{P,x}^{0})$. However, in this case
$F_P=F$ and $V_{P,x}^{0}=V$. Also $O(V)_x \cong O(V_{P,x}^{0}).$
Now, $\eps$ commutes with $x$ and hence $\eps \in O(V)_x \cong
O(V_{P,x}^{0})$. Hence
$$(O(V)_{\eps})_x \cong (O(V_{P,x}^{0}))_{\eps} < O(V_{P,x}^{0}).$$\\
Case 4. General case\\
Let $P = \prod_{i \in I} P_i$ be the decomposition of $P$ to
irreducible multiples. Since $P$ is proportional to $inv(P)$,
every $P_i$ is proportional to $P_{s(i)}$ where $s$ is some
permutation of $I$ of order $\leq 2$.

Let $I = \bigsqcup I_{\alpha}$ be the decomposition of $I$ to
orbits of $s$. Denote $V_{\alpha}:=Ker(\prod_{i \in \alpha}
P_i(x))$. Clearly $V = \bigoplus V_{\alpha}$ and $V_{\alpha}$ are
orthogonal to each other and $\eps$-invariant. Hence the pair
$(O(V)_x,(O(V)_{\eps})_x)$ is a product of pairs from the first 3
cases.

\end{proof}

\subsection{The pair $(U(V_1 \oplus V_2), U(V_1) \times U(V_2))$}


\begin{theorem} \label{CompDesU_UtU}
Let $(V,B)$ be a hermitian space over $(D,\tau)$.\\
Let $\eps \in U(V)$ be an element of order 2. Then all the
descendants of the pair $(U(V),U(V)_{\eps})$ are products of pairs
of the types\\
(i) $(GL(W) \times GL(W), \Delta GL(W))$ for some linear space $W$
over some field $D'$ that extends $D$\\
(ii) $(U(W) \times U(W), \Delta U(W))$ for some hermitian space
$W$ over some extension $(D',\tau')$ of $(D,\tau)$\\
(iii) $(GL(W),U(W))$ for some hermitian space $W$ over some
extension $(D',\tau')$ of $(D,\tau)$\\
(iv) $(GL(W_{D'}), GL(W))$ where $F'$ is a field extension of $D$,
$D'/F'$ is a quadratic extension, $W$ is a linear space over $F'$
and $W_{D'}:= W \otimes_{F'} D'$ is its extension of scalars to
$D'$ \\
(v) $(GL(W),GL(W)_{\eps})$ where $W$ is a linear space over $D$
and $\eps \in GL(W)$ is an element of order $\leq 2$.\\
(vi) $(U(W_E) , U(W))$ where $W$ is a hermitian space over some
extension $(D',\tau')$ of $(D,\tau)$, $(E,\tau'')$ is some
quadratic extension of $(D',\tau')$ and $W_E=W \otimes_{D'}E$ is
an extension of scalars with the corresponding hermitian
structure.\\
(vii) $(U(W),U(W)_{\eps})$ where $W$ is a hermitian space over
$(D,\tau)$.
\end{theorem}

For the proof of this theorem we will need the following
straightforward lemma.

\begin{lemma}
Let $(V,B)$ be a hermitian space over $(D,\tau)$.\\
Let $\eps \in U(V)$ be an element of order 2. Let $x \in U(V)$
such that $\eps x \eps = x^{-1}$. Let $Q$ be an irreducible
polynomial. Then $\eps$ gives an $D_Q$- linear isomorphism
$V_{Q,x}^{ij} \cong V_{Q,x}^{1-i,j}$.
\end{lemma}

\begin{proof}[Proof of Theorem \ref{CompDesU_UtU}]
Let $x\in U(V)^{\sigma}$ be a semi-simple element and let $P$ be
its minimal polynomial.

Note that the minimal polynomial of $x^*$ is $P^*$ and hence $P^*$
is proportional to $P$. Since $x\in U(V)^{\sigma}$, we have
$x^{-1} = \eps x \eps$ and hence its minimal polynomial is $P$.
Hence $P$ is proportional to $inv(P)$. We will now discuss 7
special cases and then deduce the general case from
them.\\\\
Case 1. $P=QQ^*inv(Q)\tau(Q)$, where $Q$ is an irreducible
polynomial.\\
Note that $GL(V)_x \cong \prod _{ij} GL(V_{Q,x}^{ij}) \cong
GL(V_{Q,x}^{00})^4 $. This identifies $U(V)_x$ with a diagonal
$\Delta GL(V_{Q,x}^{00})^2 < GL(V_{Q,x}^{00})^4$ and
$(U(V)_{\eps})_x$ with a diagonal $\Delta
GL(V_{Q,x}^{00}) < GL(V_{Q,x}^{00})^4$. \\\\
Case 2. $P=Qinv(Q)$, where $Q$ is an irreducible polynomial and
$Q^*=Q$.\\
Note that $GL(V)_x \cong \prod _{i} GL(V_{Q,x}^{i0}) \cong
GL(V_{Q,x}^{00})^2$. Note also that in this case $V_{Q,x}^{i0}$
and $V_{Q,x}^{i1}$ are identical as sets and $F$-vector spaces but
the actions of $D_Q$ on them differ by a twist by $\mu^{01}$. Now
the isomorphism $V_{Q,x}^{i0} \cong (V_{Q,x}^{i1})^*$ gives a
$(D_Q, \mu^{01})$-hermitian structure on $V_{Q,x}^{i0}$.
Therefore, $U(V)_x \cong U(V_{Q,x}^{00})\times U(V_{Q,x}^{10})$.
Note that $\eps$ gives an isomorphism of $(D_Q,
\mu^{01})$-hermitian spaces between $V_{Q,x}^{00}$ and
$V_{Q,x}^{01}$. Hence $$U(V)_x \cong U(V_{Q,x}^{00})^2 \text{ and
}
(U(V)_{\eps})_x \cong \Delta U(V_{Q,x}^{00}) < U(V_{Q,x}^{00})^2.$$\\
Case 3. $P=Qinv(Q)$, where $Q$ is an irreducible polynomial and
$Q^* = inv(Q)$.\\
Note that $GL(V)_x \cong \prod _{i} GL(V_{Q,x}^{i0}) \cong \prod
_{j} GL(V_{Q,x}^{0j})$.

Since $B$ defines a non-degenerate pairing $V_{Q,x}^{00} \cong
(V_{Q,x}^{01})^*$ and $V_{Q,x}^{0i}$ are isotropic, we have
$$U(V)_x \cong \Delta GL(V_{Q,x}^{00}) < (GL(V_{Q,x}^{00}))^2.$$
Note that in this case $V_{Q,x}^{ij}$ and $V_{Q,x}^{1-i,1-j}$ are
identical as sets and as $F$-vector spaces but the action of $D_Q$
on them differs by a twist by $\mu^{11}$.

Now, compose the isomorphism $V_{Q,x}^{00} \cong V_{Q,x}^{10}$
given by $\eps$ with the isomorphism $V_{Q,x}^{10} \cong
(V_{Q,x}^{11})^*$. This gives a $(D_Q,\mu^{11})$ unitary structure
on $V_{Q,x}^{00}$. Similarly we get a unitary structure on
$V_{Q,x}^{10}$. Finally, $\eps$ gives an isomorphism $V_{Q,x}^{00}
\cong V_{Q,x}^{10}$ as unitary spaces and hence
$$(U(V)_{\eps})_x \cong  \Delta U(V_{Q,x}^{00}) < \Delta
GL(V_{Q,x}^{00}).$$\\
Case 4. $P=QQ^{*}$, where $Q$ is an irreducible polynomial, $Q=
inv(Q)$ and $x \neq x^{-1}$.\\
Note that $GL(V)_x \cong \prod _{j} GL(V_{Q,x}^{0j})$ and as
before
$$U(V)_x \cong  \Delta GL(V_{Q,x}^{00}) < (GL(V_{Q,x}^{00}))^2.$$
In this case $V_{Q,x}^{0j}$ and $V_{Q,x}^{1j}$ are identical as
sets and as $F$-vector spaces but the action of $D_Q$ on them
differs by a twist by $\mu^{10}$. Hence $\eps$ gives a
$(D_Q,\mu^{10})$ anti-linear automorphism of $V_{Q,x}^{0j}$. Let
$W_j:=(V_{Q,x}^{0j})^{\eps}$. This is a linear space over
$(D_Q)^{\mu^{10}}$. Therefore,
$$(U(V)_{\eps})_x \cong \Delta GL(W_0) <  \Delta GL(V_{Q,x}^{00}).$$\\
Case 5. $P=QQ^{*}$, where $Q$ is an irreducible polynomial, $Q=
inv(Q)$ and $x = x^{-1}$.\\
As in the previous case, $$GL(V)_x \cong \prod _{j}
GL(V_{Q,x}^{0j}) \text{ and } U(V)_x \cong  \Delta
GL(V_{Q,x}^{00}) < (GL(V_{Q,x}^{00}))^2.$$
In this case $D_Q = D$ and $\mu^{10}$ is trivial. Hence
$V_{Q,x}^{0j}$ and $V_{Q,x}^{1j}$ are identical as $D_Q$-linear
spaces.

Also, $\eps$ gives a $D_Q$-linear automorphism of $V_{Q,x}^{0j}$.
So we can interpret $\eps$ as an element in $GL(V_{Q,x}^{00})$.
Therefore,
$$(U(V)_{\eps})_x \cong \Delta (GL(V_{Q,x}^{00}))_{\eps} <  \Delta
GL(V_{Q,x}^{00}).$$\\
Case 6. $P$ is irreducible and $x \neq x^{-1}$\\
In this case $GL(V)_x \cong GL(V_{P,x}^{00})$. Also,
$V_{P,x}^{00}$ and $V_{P,x}^{01}$ are identical as $F$-vector
spaces but the action of $D_P$ on them differs by a twist by
$\mu^{01}$. Again, the isomorphism $V_{P,x}^{00} \cong
(V_{P,x}^{01})^*$ gives a $(D_P,\mu^{01})$ hermitian structure on
$V_{P,x}^{00}$ and
$$U(V)_x \cong U(V_{P,x}^{00}).$$
Note that $V_{P,x}^{00}$ and $V_{P,x}^{10}$ are identical as
$F$-vector spaces but the action of $D_P$ on them differs by a
twist by $\mu^{10}$. Hence, $\eps$ gives a $(D_P,\mu^{10})$
anti-linear automorphism of $V_{P,x}^{00}$. Denote $W:=
(V_{P,x}^{00})^{\eps}$. It is a linear space over
$(D_P)^{\mu^{10}}$. It has a
$((D_P)^{\mu^{10}},\mu^{01}|_{(D_P)^{\mu^{10}}})$ hermitian
structure. Now
$$(U(V)_{\eps})_x \cong U(W) < U(V_{P,x}^{00}).$$\\
Case 7. $P$ is irreducible and $x = x^{-1}$.\\
Again, $$GL(V)_x \cong GL(V_{P,x}^{00}) \text{ and } U(V)_x \cong
U(V_{P,x}^{00}).$$ In this case $D_P=D$ and $\mu^{01}=\tau$. Also,
$\eps$ commutes with $x$ and hence $\eps \in U(V)_x \cong
U(V_{P,x}^{00})$. Hence
$$(U(V)_{\eps})_x \cong U(V_{P,x}^{00})_{\eps} < U(V_{P,x}^{00}).$$\\
Case 8. General case\\
Let $P = \prod_{i \in I} P_i$ be the decomposition of $P$ to
irreducible multiples. Since $P$ is proportional to $inv(P)$,
every $P_i$ is proportional to $P_{s_1(i)}$ for some permutation
$s_1$ of $I$ of order $\leq 2$. Since $P$ is proportional to
$P^*$, every $P_i$  is proportional to some $P_{s_2(i)}$. This
gives rise to an action of $\Z / 2\Z \times \Z / 2\Z$ on $I$.

Let $I = \bigsqcup I_{\alpha}$ be the decomposition of $I$ to
orbits of this action. Denote $V_{\alpha}:=Ker(\prod_{i \in
\alpha} P_i(x))$. Clearly $V = \bigoplus V_{\alpha}$ and
$V_{\alpha}$ are orthogonal to each other and $\eps$-invariant.
Hence the pair $(U(V)_x,(U(V)_{\eps})_x)$ is a product of pairs
from the first 7 cases.
\end{proof}

\subsection{Genealogical trees of the symmetric pairs considered in this paper} \label{SecDiag}
$ $\\
The following diagrams sum up the results of this section.

An arrow $"(G_1,H_1) \to (G_2,H_2)"$ means that pairs of type
$(G_1,H_1)$ may have descendants with factor of the type
$(G_2,H_2)$. We will not draw the obvious arrows $"(G,H) \to
(G,H)"$ and when we draw $"(G_1,H_1) \to (G_2,H_2) \to (G_3,H_3)"$
we mean also $"(G_1,H_1) \to (G_3,H_3)"$.

{\small
$$\framebox{\parbox{470pt}{

%
%
%
\xymatrix{
 & \framebox{\parbox{80pt}{$ \quad \,\, (U(V),U(V)_{\eps})$}}\ar@{->}[d]\ar@{->}[dl]\ar@{->}[ddr]\ar@{->}[ddrr] & &
\\
 \framebox{\parbox{64pt}{$(GL(V),U(V))$}}\ar@{->}[d] & \framebox{\parbox{80pt}{$\,\,\,\, (GL(V),GL(V)_{\eps})$}}\ar@{->}[d]\ar@{->}[dl] & &\\
\framebox{\parbox{115pt}{$(GL(V)\times GL(V), \Delta GL(V))$}} &
\framebox{\parbox{80pt}{$(GL(V)_{E/F},GL(V))$}} &
\framebox{\parbox{67pt}{$(U(V)_{E/F}, U(V))$}} &
\framebox{\parbox{92pt} {$(U(V) \times U(V), \Delta U(V))$}}}
$ $\\
\text{Here }V\text{ is a linear or hermitian space over some
finite field extension of }F \text{ and }E \text{ is} \text{some
quadratic extension of } F. }}$$
%
%
\framebox{\parbox{475pt}{
$$ \xymatrix{ \framebox{\parbox{63pt}{$(O(V),O(V)_{\eps})$}}\ar@{->}[d]\\
 \framebox{\parbox{63pt}{$(U(V_E),O(V))$}}\ar@{->}[d]\\
\framebox{\parbox{63pt}{$(GL(V), O(V))$}}}$$
$ $\\
\text{Here }V\text{ is a quadratic space over some finite field
extension of }F \text{ and }E \text{ is} \text{some quadratic
extension of } F. }}}


\end{document}